\documentclass{amsart}[35pt]
\usepackage{amsmath, amsthm, amscd, amssymb, amsfonts, inputenc, tikz}
\UseRawInputEncoding

\DeclareMathOperator{\SubF}{SubF}

\DeclareMathOperator{\Dif}{Dif}

\DeclareMathOperator{\Mid}{Mid}
\usetikzlibrary{shapes,arrows}

\usepackage[margin=3cm]{geometry}

\newtheorem{theorem}{Theorem}[section]
\newtheorem{lem}[theorem]{Lemma}

\newtheorem{cor}[theorem]{Corollary}
\theoremstyle{definition}
\newtheorem{defn}[theorem]{Definition}
\newtheorem{ex}[theorem]{Example}

\theoremstyle{remark}
\newtheorem{rem}[theorem]{Remark}
\numberwithin{equation}{section}
\begin{document}
\title[Complete algorithms for transversals and double coset representatives
]{Complete algorithms for transversals and double coset representatives of subgroups via subfactors}
\author{M.H.Hooshmand}
\address{
Department of Mathematics, Shiraz Branch, Islamic Azad University, Shiraz, Iran}

\email{\tt hadi.hooshmand@gmail.com, MH.Hooshmand@iau.ac.ir}


\date{}

\begin{abstract}
Recently, we have introduced and studied the topic of sub-indices and sub-factors of groups.
During those studies, an algorithm for obtaining the sub-factors of a finite group was stated and proved,
which has a particular case for calculating the transversals of subgroups.
  In this work, we first show that it is a complete algorithm for obtaining all transversals. Then, motivated by it,
  we state and prove a complete general algorithm to obtain all representatives
  and the number of double cosets of subgroups (which has not been possible so far).
  Moreover, we introduce the concept of middle direct product of three subsets and several
  equivalent conditions for a subset to be a complete set of representatives
  of double cosets. Also, as another important result of the topic, we provide a definitive method
  to obtain the middle factor of groups relative to a couple of subgroups.
\end{abstract}
\maketitle
\section{Introduction and preliminaries}
\noindent
Conducting a series of research has led to various results,
 one of which is finding a complete algorithm to calculate the representatives of
 double cosets for arbitrary subgroups (without any conditions), which has not been possible so far.
 Let us write a little about those conceptions, ideas, and topics.
 When we were studying the upper periodic subsets and some functional equations on algebraic structures,
  we faced the direct product of subsets of magmas, semigroups, and groups (see \cite{Hoosh5}).
It then led to the next challenging conjecture
(see \cite{Hoosh1}, The Kourovka Notebook \cite{Khu}, no. 19(2018),
Question 19.35. \& no. 20(2022), Question 20.37, and \cite{Hoosh4}).\\
{\bf Conjecture.} For every finite group $G$ and every factorization $|G|=ab$
there exist subsets $A,B\subseteq G$ with $|A|=a$ and $|B|=b$ such that $G=AB$.\\
The conjecture  is true for all finite solvable groups, finite groups such that for every divisor $d$ of their order, there is
 a subgroup of order or index $d$, and all groups of orders $\leq 10000$. Also, its generalized form ($k$-factorization)
 is true for all super solvable groups (see \cite{Hoosh4,Bil}).\\
It is worth noting that $G=A\cdot B$ in the above conjecture, where $A \cdot B$
 denotes the direct product of subsets $A,B$ which means the representation of every
  element of $AB$ by $x=ab$  with $a\in A$, $b\in B$ is unique.
Hence, $G=A \cdot B$ if and only if $G=A B$ and the product $A B$ is direct
(a factorization of $G$ by two subsets, the additive notation is $G=A\dot{+} B$).
We call $A$ a left
factor of $G$ if and only if $G=A \cdot B$ for some $B\subseteq G$
(equivalently, there exists a right factor $B$ of $G$ relative to $A$).
For example, every subgroup is a left (resp. right) factor
relative to its right (resp. left) transversal, hence it is a two sided factor of $G$.\\
Since the product $A\{b\}$ is direct for all $b\in G$, the question that arises is how much
 can we extend this direct product? And what happens if $A$ is a left factor of $G$?
Then, we arrived at an important conception that is a generalization of factors of groups, namely sub-factors.
Also, we introduced sub-indices of arbitrary subsets of groups which can be considered as some generalizations
of the subgroups' indices (see \cite{Hoosh2, Hoosh3}). The topic has some results (for subgroups) such as the next lemma.
\begin{lem}
Let $G$ be a group, $H\leq G$ and $T\subseteq G$. Then, the following statements are equivalent:\\
$($a$)$ $T$ is a right transversal of $H$ in $G$;\\
$($b$)$ $G=H\cdot T$;\\
$($c$)$ The product $HT$ is direct and $T$ is maximal relative to the direct product $($i.e., $T$ is a right sub-factor of $G$ relative to $H$ \cite{Hoosh2},
which means $HT=H\cdot T$ and if $T\subset S$, then the product $HS$ is not direct$)$.
\end{lem}
\begin{proof}
This is immediately concluded from Corollary 3.9 of \cite{Hoosh2}, by taking $A:=H\leq G$.
\end{proof}
{\bf A complete algorithm for finding transversals of subgroups.} In \cite{Hoosh3} we have introduced
a proven algorithm for obtaining sub-factors of arbitrary subsets of groups.
Now, we can obtain a complete algorithm for right (and similarly left) transversals of subgroups from it,
namely, ``Right Transversal Algorithm'' (RTA) .
\begin{theorem}[RTA]
Let $G$ be a finite group and $H\leq G$. Then, put $H^c:=G\setminus H$,  $\mathcal{C}^{(-1)}(H):=G$,
fix $g_0\in G$, and
construct the sequences $\{g_n\}_{n\geq 0}$ and $\{\mathcal{C}^{(n)}(H)\}_{n\geq-1}$ as follows:\\
By the assumption that $g_0,\cdots , g_{n}$ and
$\mathcal{C}^{(-1)}(H),\cdots , \mathcal{C}^{(n-1)}(H)$ are defined, for an integer
 $n\geq 0$, set
\begin{equation}
\mathcal{C}^{(n)}(H):=\mathcal{C}^{(n-1)}(H)\cap H^cg_{n},
\end{equation}
 and then choose an  element $g_{n+1}$ in
$\mathcal{C}^{(n)}(H)$ if it is nonempty, and also put $B_n:=\{g_0,\cdots , g_{n+1}\}$
$($thus $B_{-1}=\{g_0\}$, $B_{n-1}\cup \{g_{n+1}\}=B_n$ and $\mathcal{C}^{(n)}(H)=\bigcap_{i=0}^{n}H^cg_{i}$ for all $n\geq 0)$.\\
Then there exists a least integer $N\geq 0$ such that $\mathcal{C}^{(N)}(H)=\emptyset$, and we have $|G:H|=N+1$,
$B:=B_{N-1}$ $($with $N+1$ elements$)$ is a right transversal of $H$ in $G$.\\
Moreover, every right transversal of $H$ can be obtained from this algorithm $($i.e., it is complete$)$.
\end{theorem}
\begin{proof}
The first part of the proof is obtained form Theorem 4.1 of \cite{Hoosh3} and Lemma 1.1, by taking $A:=H\leq G$.\\
Now, let $X$ be an arbitrary right transversal of $H$ in $G$ and represent its members by $X=\{x_0,\cdots , x_{m}\}$ where $m=|X|-1$
(thus $m\geq 0$). In the algorithm choose $g_0:=x_0$ (since $g_0$ is arbitrary).
If $H=G$, then $|X|=1$, $N=0$ and
$X=\{x_0\}=B=B_{-1}$ thus we are done. Otherwise, suppose that $g_0,\cdots,g_n$ take
the values $x_0,\cdots,x_n$, respectively,  for
$n<m$. Since the product $H (\{x_0,\cdots , x_{n}\}\cup \{x_{n+1}\})$ is direct,
$x_{n+1}\neq hx_i$, for all $h\in H$ and $i=0,\cdots,n$. Hence
$$
x_{n+1}\in \bigcap_{i=0}^nH^cx_i=\bigcap_{i=0}^nH^cg_i=\mathcal{C}^{(n)}(H)
$$
So $g_{n+1}$ can take the value $x_{n+1}$ in the algorithm process.
Therefore $X\subseteq B_{N-1}$ and Lemma 1.1 requires $X=B_{N-1}$ (and so $N=m$).
\end{proof}

\begin{ex}
Consider the additive group $G:=\mathbb{Z}_{12}$, $H=\{0,3,6,9\}\leq G$, and $g_0:=0$ in RTA.
Since $g_1\in H^c$ we can choose $g_1:=1$. Then
$$
g_2\in H^c\cap (H^c+1)=\{2,5,8,11\}.
$$
If $g_2:=2$, then $\mathcal{C}^{(2)}(H)=\{2,5,8,11\}\cap (H^c+2)=\emptyset$. Therefore,
$|G:H|=N+1=3$ and $T=\{0,1,2\}$ is a transversal for $H$ in $G$.
Note that by making all possible selections for $g_n$'s (in RTA), all tranversals of $H$ can be obtained.
\end{ex}
\section{Complete algorithms for representatives of double cosets and middle factors}
Motivated by RTA and Lemma 1.1, we are thinking about some equivalent
conditions for a subset to be a complete set of representatives of double cosets of a couple of
subgroups, and also
a complete algorithm that enables us to obtain all such representatives.
For this purpose, we need the direct product of three subsets $AXB$, denoted by $A\cdot X\cdot B$,
that is defined similar to the direct product of two subsets (see \cite{Hoosh3}). But, during recent studies,
we have found that there is another conception weaker than $A\cdot X\cdot B$
which plays an important role regarding representatives of double cosets.
\begin{defn}
We call the product $AXB$ middle direct (or direct relative to $X$) and denote it
by $A\dot{X}B$  if the equality $axb=a'x'b'$
requires $x=x'$, for all $a,a'\in A$, $x,x'\in X$ and $b,b'\in B$.
\end{defn}
If the product $AXB$ is direct then it is middle direct but the converse is not true.
For instance $A\{x\}B$ is always middle direct (for all $x$ in $G$) and it is not direct
in general (e.g., consider $x=a^2$, $A:=H$ and $B:=K$ in Example 2.14 at the end of this paper). Hence, we put
$$\Mid_G(A,B):=\{g\in G: AgB=A\cdot \{g\}\cdot B\},$$
and call it middle director of $(A,B)$ in $G$. Since $A\cdot X\cdot B$ implies
$A\cdot x\cdot B$, for all $x\in X$, and  the product $A\emptyset B$ is always direct (here $X$
is empty),
\begin{equation}
\Mid_G(A,B)=\emptyset \Leftrightarrow  AXB=A\cdot X\cdot B \mbox{ requires }X=\emptyset \mbox{, for every subset $X$}
\end{equation}
Note that the product $AB$ is direct if and only if $1\in \Mid_G(A,B)$, and if $AB=A\cdot B$, then
$$
\{x\in G: A\subseteq A^x\}\cup \{x\in G: B^x\subseteq B\}\subseteq \Mid_G(A,B),
$$
where $A^x=xAx^{-1}$.
Hence, if $A$ or $B$ is a normal subset of $G$, then
$$
\Mid_G(A,B)\neq\emptyset
\Leftrightarrow \Mid_G(A,B)=G \Leftrightarrow  AB=A\cdot B\Leftrightarrow 1\in \Mid_G(A,B)
$$
For if $A$ is a normal subset of $G$, $x_0\in \Mid_G(A,B)\neq\emptyset$ and $a_1xb_1=a_2xb_2$
for $a_1,a_2\in A$, $x\in G$ and $b_1,b_2\in B$, then there exist $a_1',a_2'\in A$ such that
$xa_1'b_1=xa_2'b_2$ and so there are $a_1'',a_2''\in A$ such that $a_1''x_0b_1=a_2''x_0b_2$
which implies $b_1=b_2$. Therefore, the product $AxB$ is direct and hence $x\in \Mid_G(A,B)$.
Thus, in abelian groups, $\Mid_G(A,B)$ is either empty or the whole $G$.
It is easy to check that
\begin{equation}
AXB=A\cdot X\cdot B \Leftrightarrow AXB=A\dot{X}B \mbox{ and } X\subseteq \Mid_G(A,B)
\end{equation}
 for every subsets $A,B,X$.\\
It is obvious that if $H,K$ are subgroups of $G$, then
\begin{equation}
H\{x\}K \mbox{ is direct }\Leftrightarrow x\in \Mid_G(H,K)\Leftrightarrow H\cap K^x=\{1\}\Leftrightarrow H^{x^{-1}}\cap K=\{1\}
\Leftrightarrow Hx\cap xK=\{x\}
\end{equation}
Now, we prove some (basic) equivalent conditions for a subset $X$ to be a complete set of representatives of double cosets
in arbitrary groups.
\begin{lem}
Let $H,K\leq G$ and $X\subseteq G$. Then, the following statements are equivalent:\\
$($a$)$ $X$ is a complete set of representatives of double cosets of $(H,K)$ in $G$;\\
$($b$)$ $G=H\dot{X}K$;\\
$($c$)$ $HXK=H\dot{X}K$ and $X$ is maximal relative to the middle direct product.
\end{lem}
\begin{proof}
First note that if  $HXK=H\dot{X}K$, then
\begin{equation}
\big(\forall g\in X^c\big)\Big(H (X\cup\{g\})K\mbox{ is middle direct }\Leftrightarrow
g\in (HXK)^c\Big)
\end{equation}
For if $g\in HXK$ and $g\neq x$ for all $x\in X$, then $HgK=HxK$  and so $H (X\cup\{g\})K$ is not middle direct.
Conversely, if the product $H (X\cup\{g\})K$ is not middle direct, then the assumption $HXK=H\dot{X}K$
requires that $h_1gk_1=h_2xk_2$ for some $h_1,h_2\in H$, $x\in X$ and $k_1,k_2\in K$ which means
$g\in HxK\subseteq HXK$.\\
Now, if (a) holds and $h_1x_1k_1=h_2x_2k_2$ for $h_1,h_2\in H$, $x_1,x_2\in X$ and $k_1,k_2\in K$,
then $x_1\in Hx_2K_2$ and so $x_1=x_2$ (by the definition of $X$). Hence we arrive at (b).\\
Also, if (b) is satisfied, then $(2.4)$ implies that $H (X\cup\{g\})K$ is not middle direct for every
$g\in X^c$ (because $(HXK)^c=\emptyset$). Thus (c) is obtained.
Finally, if (c) is true, then all members of the family  $\{HxK:x\in X\}$ are disjoint and $G=HXK$
(for if $g\in (HXK)^c$ then we get a contradiction by $(2.4)$), and hence
the proof is complete.
\end{proof}
\begin{rem}
Due to the sub-factor and sub-index theory of groups and  the above definition and lemma,
we propose the title ``middle transversal'' instead of ``a complete set of representatives of double cosets''.
Thus, $X$ is a middle transversal of a couple of subgroups $(H,K)$ if and only if $G=H\dot{X}K$
(we have introduced and studied middle sub-factors of an arbitrary couple
of subsets in another paper, as a continuation of the sub-factor theory of groups in \cite{Hoosh2, Hoosh3}).
This agrees with the concepts of left and right transversals, sub-factors, factors, and related factorization
in groups (see \cite{Hoosh2, Hoosh3}).
\end{rem}
Now, we are ready to establish the first complete algorithm for obtaining
complete sets of double cosets representatives of subgroups, namely the ``Middle Transversal Algorithm'' (MTA), without
any conditions. One can get some information about finding double coset representatives from \cite{Holt} (page 131).
\begin{theorem}[MTA]
Let $G$ be a finite group and $H,K$ subgroups of $G$.
Put  $\mathcal{C}^{(-1)}(H,K):=G$ and fix $g_0\in \mathcal{C}^{(-1)}(H,K)$.
Then, construct the finite sequences $\{g_n\}_{n\geq 0}$ and $\{\mathcal{C}^{(n)}(H,K)\}_{n\geq-1}$ as follows:\\
By the assumption $g_0,\cdots , g_{n}$ and $\mathcal{C}^{(-1)}(H,K),\cdots , \mathcal{C}^{(n-1)}(H,K)$ are defined, for an integer
 $n\geq 0$, set
\begin{equation}
\mathcal{C}^{(n)}(H,K):=\mathcal{C}^{(n-1)}(H,K)\cap (Hg_nK)^c,
\end{equation}
 and then choose an  element $g_{n+1}$ in
$\mathcal{C}^{(n)}(H,K)$ if it is nonempty, and also put $X_n:=\{g_0,\cdots , g_{n+1}\}$
$($thus $X_{-1}=\{g_0\}$, $X_{n-1}\cup \{g_{n+1}\}=X_n$, and
$\mathcal{C}^{(n)}(H,K)=\bigcap_{i=0}^{n}(Hg_{i}K)^c$ for all $n\geq 0)$.\\
Then there exists a least integer $N\geq 0$ such that $\mathcal{C}^{(N)}(H,K)=\emptyset$
$($equivalently $H X_{N-1} K=G$$)$, the number of double cosets is $N+1$,
and $X:=X_{N-1}$ $($with $N+1$ elements$)$ is a complete set of double coset representatives of $(H,K)$ in $G$.
Moreover, every middle transversal of $(H,K)$ can be gotten from this algorithm.
\end{theorem}
\begin{proof}
If $Hg_{0}K=G$, then putting $N=0$ and
$X=X_{-1}=\{g_0\}$, the conditions hold.
Now, let $Hg_{0}K\neq G$, then:\\
(a) All elements $g$ in this algorithm are distinct since $g_{n+1}\notin Hg_{i}K\supseteq \{g_i\}$ for every $0\leq i\leq n$.
\\
(b) If $\mathcal{C}^{(n)}(H,K)\neq\emptyset$, for some $n\geq 0$, then the product $HX_nK$ is middle direct.
 Because the product $HX_0K$ is middle direct and
if $HX_iK$ is middle direct, for all $0\leq i\leq n-1$, then using $(2.4)$ we have
$$
H X_{n}K\mbox{ is middle direct }\Leftrightarrow H (X_{n-1}\cup\{g_{n+1}\})K\mbox{ is middle direct }
$$
$$
\Leftrightarrow
g_{n+1}\notin HX_{n-1}K\Leftrightarrow g_{n+1}\in (HX_{n-1}K)^c=\mathcal{C}^{(n)}(H,K)
$$
Also, $(2.5)$ implies
$\mathcal{C}^{(n)}(H,K)\subset \mathcal{C}^{(n-1)}(H,K)$, since
$$
g_n\in \mathcal{C}^{(n-1)}(H,K)\neq\emptyset\; , \;
g_n\notin \bigcap_{i=0}^{n}(Hg_{i}K)^c=\mathcal{C}^{(n)}(H,K)
$$
Therefore, we have the following strictly decreasing chain of subsets
$$
\mathcal{C}^{(-1)}(H,K)\supset \mathcal{C}^{(0)}(H,K)\supset \cdots
$$
that surly ends to the empty since $G$ is finite.\\
(c) There exists a least non-negative integer $N$ such that $\mathcal{C}^{(N)}(H,K)=\emptyset$. Indeed
\begin{equation}
G=\mathcal{C}^{(-1)}(H,K)\supset \cdots \supset\mathcal{C}^{(N)}(H,K)=\emptyset.
\end{equation}
We claim that $X:=X_{N-1}$
is a set of double coset representatives of $(H,K)$.
Because the product $HX_{N-1}K$ is middle direct, by part(b), and if $g\in X^c_{N-1}$, then
$$
H (X_{N-1}\cup\{g\})K \mbox{ is middle direct } \Leftrightarrow
g\in \bigcap_{x\in X_{N-1}}(HxK)^c=\bigcap_{i=0}^{N}(Hg_{i}K)^c=\mathcal{C}^{(N)}(H,K)
=\emptyset
$$
Hence the claim is proved due to $(2.4)$ and Lemma 2.2.\\
(d) Now, consider an arbitrary middle transversal of $G$ relative to $(H,K)$ that we assume as
$\chi=\{x_0,\cdots , x_{m}\}$, for some integer $m\geq 0$. Choose $g_0:=x_0$ in MTA.
Suppose that  $g_0:=x_0,\cdots , g_n:=x_n$, for $0\leq n\leq m-1$ (if $m\geq 1$). Since $H (\{x_0,\cdots , x_{n}\}\cup\{x_{n+1}\})K$
is middle direct
and $x_{n+1}\in \{x_0,\cdots , x_{n}\}^c$, we have
$x_{n+1}\in \bigcap_{i=0}^{n}(Hx_{i}K)^c\neq\emptyset$ (by $(2.4)$). Thus one can choose
$g_{n+1}:=x_{n+1}$ in the algorithm. This shows that $m\leq N$. But the maximality of
$\chi=\{x_0,\cdots , x_{m}\}$ respect to the middle direct product requires that $\chi=X_{N-1}$ and $N=m$
(by Lemma 2.2).
\end{proof}
For a useful example for MTA, $G$ should be non-abelian and neither $H$ nor $K$ be normal.
Otherwise, it is like we have done RTA for the subgroup $HK$ and so $N+1=|G:HK|$.
\begin{ex}
Let $$G:=D_{12}\cong \langle a,b|a^6=1=b^2 , bab=a^{-1}\rangle=\{1,a,a^2,a^3,a^4,a^5,b,ba,ba^2,ba^3,ba^4,ba^5\},$$
and consider the subgroups $H=\{1,a^3,ba^3,b\}$ and $K=\{1,a^3,ba,ba^4\}$. Choosing $g_0=1$
in MTA, we have $Hg_0K=HK=\{1,a,a^2,a^3,a^4,b,ba,ba^3,ba^4\}$ and so $g_1\in \mathcal{C}^{(0)}(H,K)=(HK)^c=\{a^2,a^5,ba^2,ba^5\}$.
If put $g_1=a^2$, then $Hg_1K=\{a^2,a^5,ba^2,ba^5\}=(HK)^c$ and hence $\mathcal{C}^{(1)}(H,K)=(HK)^c\cap (Ha^2K)^c=\emptyset$.
Therefore, the number of double cosets is $N+1=2$ and $X=X_0=\{1,a^2\}$ is a complete set of double coset representatives of $(H,K)$,
and we have the following middle direct (product) representation for $D_{12}$
  $$D_{12}=\{1,a^3,ba^3,b\}\dot{\{1,a^2\}}\{1,a^3,ba,ba^4\},$$
although the product is not direct (this also shows that $\Mid_G(H,K)\neq D_{12}$).
Thus $\{1,a^5\}$, $\{1,ba^2\}$ and $\{1,ba^5\}$ are also middle transversals of
$(\{1,a^3,ba^3,b\},\{1,a^3,ba,ba^4\})$ in $D_{12}$.

Now, for calculating $\Mid_G(H,K)$, note that
$H\cap K=\{1,a^3\}= Z(D_{12})$ implies $Hx\cap xK\supseteq \{x,a^3x=xa^3\}$, for all $x\in G$. Therefore
$\Mid_G(H,K)=\emptyset$.
\end{ex}
Up to now, we have introduced some equivalent conditions and a complete algorithm for finding a subset $X$ such that
$G=H\dot{X}K$.
Now, some essential questions arise as follows:\\
($Q_1$) Are there some necessary and sufficient conditions for such $X$ to be a middle factor of $G$ relative to $(H,K)$
(i.e., $G=H\cdot X\cdot K$)?\\
($Q_2$) Is there a complete algorithm for obtaining all middle factors?\\
($Q_3$) Can we extend every middle factor to a middle transversal?\\
First of all, we observe that $\Mid_G(H,K)$ plays a basic role for answering the questions.
\begin{lem} For all subgroups $H,K$ of an arbitrary group $G$,
the followings are equivalent:\\
$($a$)$ $\Mid_G(H,K)=G$;\\
$($b$)$ $G=H\cdot X\cdot K$ for every complete set of  representatives of double cosets $X$ of $(H,K)$ in $G$;\\
$($c$)$ $G=H\cdot X\cdot K$ for some complete set of  representatives of double cosets $X$ of $(H,K)$ in $G$;\\
$($d$)$  $G=H\cdot X\cdot K$ for some $X\subseteq G$ $($i.e. there is a middle factor of $G$ relative to $(H,K))$.\\
\end{lem}
\begin{proof}
The part (a) implies (b), by $(2.2)$ and Lemma 2.2 , and $(b)\Rightarrow (c) \Rightarrow(d)$ is obvious.
Now that (d) is satisfied and $g\in G$, then  $G=HXK$ requires $g=h_gx_gk_g$ for some
$h_g\in H$, $x_g\in X$ and $k_g\in K$. If $h_1gk_1=h_2gk_2$ for
 $h_1,h_2\in H$ and $k_1,k_2\in K$, then
 $
(h_1h_g)x_g(k_gk_1)=(h_2h_g)x_g(k_gk_2)
 $
 and
 $H\cdot X\cdot K$ implies that $h_1h_g=h_2h_g$ and $k_gk_1=k_gk_2$ which means $g\in\Mid_G(H,K)$.
 This ends the proof.
\end{proof}
\begin{rem} The above important lemma together with $(2.1)$ show that there are three main cases for $\Mid_G(H,K)$ as follows:\\
(Case 1)  $\Mid_G(H,K)=\emptyset$: we do not have any non-trivial direct product $H\cdot X\cdot K$.\\
(Case 2) $\Mid_G(H,K)=G$: not only we have middle factors $X$ of $G$ relative to $(H,K)$ (i.e., $G=H\cdot X\cdot K$;)
but also the next lemma characterizes all such subsets $X$.\\
(Case 3) $\emptyset\subset \Mid_G(H,K)\subset G$: the group $G$ must be non-abelian, neither $H$ nor $K$
be normal in $G$, and $\gcd(|H|,|K|)\neq 1$ if $G$ is finite
 (see Corollary 2.9). In this case, there are non-empty maximal subsets $X$ such that $HXK=H\cdot X\cdot K$ but
$HXK\neq G$, and so $X$ is not surely a complete set of  representatives of double cosets of $(H,K)$ in $G$.
Of course, we hope to extend such subset $X$ to a $X_*$ with $G=H\dot{X_*}K$ (see Remark 2.3).\\
Hence, we first should find maximal (non-empty) subsets $X$ relative to the property $HXK=H\cdot X\cdot K$, and then
discuss necessary and sufficient conditions for occurring the equality $HXK=G$.
\end{rem}
The next lemma introduces some equivalent conditions for a subset to be a middle factor relative
to a couple of subgroups.
\begin{lem}
Let $H,K\leq G$ and $X\subseteq G$. Then, the following statements are equivalent:\\
$($a$)$ $G=H\cdot X\cdot K$;\\
$($b$)$ $X$ is a complete set of  representatives of double cosets of $(H,K)$ in $G$ and $HXK=H\cdot X\cdot K$;\\
$($c$)$ $G=H\dot{X}K$ and $\Mid_G(H,K)=G$;\\
$($d$)$ $HXK=H\cdot X\cdot K$, $X$ is $($non-empty and$)$ maximal relative to the direct product, and $X^c\subseteq\Mid_G(H,K)$;\\
$($e$)$ $HXK=H\dot{X}K$, $X$ is maximal relative to the middle direct product, and $\Mid_G(H,K)=G$.
\end{lem}
\begin{proof}
This is a direct result of Lemma 2.6 and Lemma 2.2, by using the next facts.
If $G=H\cdot X\cdot K$, then $X$ is $($non-empty and$)$ maximal relative to the direct product $H\cdot X\cdot K$. The converse is
true if $X^c\subseteq\Mid_G(H,K)$ (but not in general, of course, note that $HXK=H\cdot X\cdot K$ implies $X\subseteq\Mid_G(H,K)$).
\end{proof}
\begin{cor}
If $H,K$ are subgroups of a finite group $G$ satisfying one of the next
conditions:\\
$($i$)$ $\gcd(|H|,|K|)=1$;\\
$($ii$)$ $H\cap K=\{1\}$, and $H\unlhd G$ or $K\unlhd G$,\\
then $\Mid_G(H,K)=G$ and so
$$
\mbox{$X$ is a complete set of  representatives of double cosets of $(H,K)$}\Leftrightarrow
G=H\cdot X\cdot K \Leftrightarrow G=H\dot{X}K
$$
\end{cor}
Now, we are ready to introduce and prove a complete algorithm, namely ``Middle Sub-Factor Algorithm (relative to subgroups)''
 (MSFA) for obtaining all the following items:\\
(I) All non-empty maximal subsets $X$ relative to the direct product $H\cdot X\cdot K$, i.e.,
middle sub-factors of $G$ relative to a couple of subgroups.\\
(II) All middle factors of $G$ relative to a couple of subgroups if any exist (i.e., all subsets $X$ satisfying
$G=H\cdot X\cdot K$).\\
(III) All complete sets of representatives of double cosets of a couple of subgroups (i.e.,
all middle transversals) which also are middle factors of $G$ relative to a couple of subgroups).\\
Of course, due to $(2.1)$ and Remark 2.7, the necessary condition for existence of such subsets $X$ is
$\Mid_G(H,K)\neq \emptyset$ (if $\Mid_G(H,K)=\emptyset$, then $\emptyset$ is the only middle sub-factor of $G$ relative to
$(H,K)$). Also, we state a lemma after the following theorem that mentions several necessary and sufficient conditions for
the out put $X=X_{N-1}$ of the algorithm to be a middle factor of $G$.

\begin{theorem}[MSFA]
Let $G$ be a finite group and $H,K$ subgroups of $G$ such that $\Mid_G(H,K)\neq\emptyset$.
Put  $\mathcal{C}^{(-1)}(H,K):=\Mid_G(H,K)$ and fix $g_0\in \mathcal{C}^{(-1)}(H,K)$.
Then, construct the sequences $\{g_n\}_{n\geq 0}$, $\{\mathcal{C}^{(n)}(H,K)\}_{n\geq-1}$, and
 $X_n$ by the relation $(2.5)$. Then, there is the least integer $N\geq 0$ such that $\mathcal{C}^{(N)}(H,K)=\emptyset$
 $($equivalently $H X_{N-1} K=\Mid_G(H,K))$,
 and put $X:=X_{N-1}$.
Also, the product $HXK$ is direct and $X$ is $($non-empty and$)$ maximal relative to $H\cdot X\cdot K$.
Moreover, every maximal subset $\chi$ relative to the direct products $H\cdot \chi\cdot K$ can be obtained from this algorithm.
\end{theorem}
\begin{proof}
Theorem 2.4 (MTA) guarantees the existence of such number $N$ and requires that the product $HXK$ is direct.
Because $\mathcal{C}^{(n)}(H,K)$ in this algorithm is indeed as the intersection of $\mathcal{C}^{(n)}(H,K)$ in MTA
and $\Mid_G(H,K)$, and $X_n\subseteq \Mid_G(H,K)$ (see $(2.2)$).
Now, if $X:=X_{N-1}$ is not maximal relative to $H\cdot X\cdot K$,
then there exists $g\in X^c$ such that the product  $H(X\cup\{g\})K$ is direct, and so $g\in (HXK)^c$
(by $(2.4)$ and $(2.2)$). This means
$$
g\in \bigcap_{x\in X}(HxK)^c=\bigcap_{i=0}^{N}(Hg_{i}K)^c=\mathcal{C}^{(N)}(H,K)
=\emptyset
$$
that is a contradiction.\\
 Now, let
$\chi=\{x_0,\cdots , x_{m}\}$, for some integer $m\geq 0$, be maximal relative to the direct product $H\cdot \chi\cdot K$.
One can choose $g_0:=x_0$ in the algorithm, because $\chi\subseteq \Mid_G(H,K)$.
Assume that  $g_0:=x_0,\cdots , g_n:=x_n$, for $0\leq n\leq m-1$ (if $m\geq 1$). Since $H (\{x_0,\cdots , x_{n}\}\cup\{x_{n+1}\})K$
is direct and $x_{n+1}\in \{x_0,\cdots , x_{n}\}^c$, $(2.4)$ and $(2.2)$ imply that
$x_{n+1}\in \bigcap_{i=0}^{n}(Hx_{i}K)^c\neq\emptyset$. Thus one can choose
$g_{n+1}:=x_{n+1}$ in the algorithm. This shows that $m\leq N$. But the maximality of
$\chi=\{x_0,\cdots , x_{m}\}$, relative to the direct product, requires that $\chi=X_{N-1}$ and $N=m$.
\end{proof}
The following lemma states that if $\Mid_G(H,K)=G$, then MTA and MSFA are the same for
$(H,K)$, and can be also considered as MFA (i.e., middle factor algorithm). Moreover,
$\Mid_G(H,K)=G$ is the necessary and sufficient condition for existence of the items (II), (III)
in the explanation before MSFA.
\begin{lem}
With the assumptions of MSFA, the followings are equivalent:\\
$($a$)$ $G=H X_{N-1} K$;\\
$($b$)$ $X_{N-1}$ is a complete set of representatives of double coset of $(H,K)$;\\
$($c$)$ $X_{N-1}$ is a middle factor of $G$ relative to $(H,K)$ (i.e., $G=H\cdot X_{N-1}\cdot K$);\\
$($d$)$ $\Mid_G(H,K)=G$;\\
$($e$)$ There exists a middle factor of $G$ relative to $(H,K)$, and every middle factor can be obtained from this algorithm;\\
$($f$)$ Every complete set of representatives of double coset of $(H,K)$ can be obtained from this algorithm;\\
$($g$)$ The set of all outputs of MTA and MSFA are equal.\\
Hence, if $\gcd(|H|,|K|)=1$ $($or the conditions of Corollary 2.9 hold$)$, then all
the above properties hold.
\end{lem}
\begin{proof}
This is a result of relation $(2.2)$, Lemma 2.6, MTA and MSFA.
\end{proof}
\begin{cor}
If $G$ is abelian, then, for all subgroups $H,K$, we have only the following two cases:\\
$($i$)$ $H\cap K=\{1\}$: all equivalent conditions of Lemma 2.6 are satisfied.\\
$($ii$)$ $H\cap K\neq \{1\}$: MSFA is not applicable, there is not any middle factor of $G$ relative to $(H,K)$,
and only MTA is applicable.
\end{cor}
The next remark answers the question $(Q_3)$ by using MSFA and MTA (for finite groups, if $G$ is infinite, then
one may use the Zorn's lemma).
\begin{rem}
Consider that we use MSFA, and $X=X_{N-1}$ is an output of it. Then, we face two possible situations:\\
(1) $G=HXK$: we have all the equivalent conditions of Lemma 2.11.\\
(2) $G\neq HXK$: we claim that continuing the algorithm process via the following way, one can reach a complete set of double coset representatives
$X_*\supset X$:\\
Put $\mathcal{C}^{(N)}_*(H,K):=G\setminus HX_{N-1}K$ (instead of $\mathcal{C}^{(N)}(H,K)=\emptyset$), and continue the algorithm process via
relation $(2.5)$ for all $n\geq N+1$ (of course, replacing $\mathcal{C}^{(n)}(H,K)$ by $\mathcal{C}^{(n)}_*(H,K)$). Then there is the least integer
$N_*\geq N+1$ such that $\mathcal{C}^{(N_*)}(H,K)=\emptyset$
 (equivalently $H X_{N_*-1} K=G$), and $X_*:=X_{N_*-1}$ $($with $N_*+1$ elements$)$ is a complete set of double coset representatives of $(H,K)$.
 Because $G\setminus HX_{N-1}K=\bigcap_{i=0}^{N}(Hg_{i}K)^c$ that is the same $\mathcal{C}^{(N)}(H,K)$ in MTA
(although $\mathcal{C}^{(N)}(H,K)=\emptyset$ in MSFA).
\end{rem}
\begin{ex}
If we change $H,K$ in Example 2.5 to
$H=\{1,ab\}$ and $K=\{1,a^3,b,ba^3\}$, then the product $HK$ is direct,
$H\cdot K=\{1,a,a^3,a^4,b,ab,a^3b,a^4b\}\subseteq \Mid_G(H,K)$ and so $(HK)^c=\{a^2,a^5,a^2b,a^5b\}$. It is interesting to know
that $Hx\subset xK$, for all $x\in (HK)^c$. Therefore, $\emptyset\neq \Mid_G(H,K)=HK\neq G$ and one can
use MSFA for this case. Note that $$\mathcal{C}^{(0)}(H,g_0,K)=(Hg_0K)^c\cap \Mid_G(H,K)=\emptyset,$$
for all $g_0\in \mathcal{C}^{(-1)}(H,K)=\Mid_G(H,K)$. Therefore, the product $H\{x\}K$ is maximally direct
(i.e., $\{x\}$ is a middle factor of $D_{12}$ relative to $(H,K)$) if
and only if $x\in \Mid_G(H,K)=HK$. Now, due to Remark 2.13, we can extend such singletons $\{x\}$ to
complete sets of representatives of double cosets (middle transversals). Here $N=0$ and considering $x=g_0=1$,
we have
$$\mathcal{C}^{(0)}_*(H,K):=G\setminus HX_{-1}K=(HK)^c=\{a^2,a^5,a^2b,a^5b\}.$$
Now, if $g_1:=a^2$, then $Ha^2K=(HK)^c$ and so
$$\mathcal{C}^{(1)}_*(H,K)=\mathcal{C}^{(0)}_*(H,K)\cap HK=\emptyset.$$
Therefore $N_*=1$, $X_{N_*-1}=\{1,a^2\}$, and we have
the following middle direct representation: $$D_{12}=\{1,ab\}\dot{\{1,a^2\}}\{1,a^3,b,ba^3\}.$$
Thus $\{1,a^5\}$, $\{1,a^2b\}$ and $\{1,a^5b\}$ are also middle transversals of the couple of subgroups
$$(\{1,ab\},\{1,a^3,b,ba^3\}).$$
\end{ex}


\begin{thebibliography}{8}
\bibitem{Bil}
 R. Bildanov, V. Goryachenko, A.V. Vasil'ev,
 {\it Factoring nonabelian finite groups into two subsets}, Sib. Èlektron. Mat. Izv., 17(2020), 683-689.
 \bibitem{Holt}
 D.F. Holt, B. Eick, E.A. \'{O}Brien, {\it Handbook of Computational Group Theory},
 Chapman and Hall/CRC, New York, 2005.
\bibitem{Hoosh1}
M.H. Hooshmand, Factor subset of finite group, 2014, https://mathoverflow.net/questions/155986/
factor-subset-of-finite-group
\bibitem{Hoosh2}
M.H. Hooshmand, {\it Index, sub-index and sub-factor of groups with interactions to number theory},
J. Alg. Appl., 19:6(2020), 1-23.
\bibitem{Hoosh3}
M.H. Hooshmand, {\it Subindices and subfactors of finite groups},
Commun. Algebra,  51:6(2023), 2644-2657.
\bibitem{Hoosh4}
M.H. Hooshmand, {\it Basic results on an unsolved problem about factorization of finite groups},
Commun. Algebra, 49:7(2021), 2927-2933.
\bibitem{Hoosh5}
M.H. Hooshmand, {\it Upper and Lower Periodic Subsets of Semigroups},
Alg. Colloq. 18:3(2011), 447-460.
\bibitem{Khu}
E.I. Khukhro and V.D. Mazurov, {\it Unsolved Problems
in  Group Theory: The Kourovka Notebook}, no. 19\& 20, Sobolev Institute of Mathematics, 2018 \& 2022.
\end{thebibliography}
\end{document}